\documentstyle[11pt,twoside,amsthm]{article}
\theoremstyle{definition}
\newtheorem{definition}{Definition}[section]
\newtheorem{notation}[definition]{Notation}

\theoremstyle{plain}
\newtheorem{theorem}[definition]{Theorem}
\newtheorem{lemma}[definition]{Lemma}
\newtheorem{proposition}[definition]{Proposition}

\newtheorem{remark}[definition]{Remark}

\newcommand{\beq}{\begin{equation}}
\newcommand{\eeq}{\end{equation}}

\begin{document} 
\title{Further Functorial Properties of the Reticulation}
\author{Claudia MURE\c{S}AN\\University of Bucharest\\Faculty of Mathematics and Computer Science\\Academiei 14, RO 010014, Bucharest, Romania\\c.muresan@yahoo.com}
\maketitle

\begin{center}
{\em Dedicated to the memory of Professor Lauren\c{t}iu Panaitopol}
\end{center}

\begin{abstract}
In this article we prove a set of preservation properties of the re\-ti\-cu\-la\-tion functor for residuated lattices (for instance preservation of subalgebras, finite direct products, inductive limits, Boolean po\-wers) and we transfer certain pro\-per\-ties between bounded distributive lattices and residuated lattices through the reticulation, focusing on Stone, strongly Stone and $m$-Stone algebras.

\end{abstract}

\section{Introduction}

\hspace*{11pt} In this paper we continue the study we began in \cite{eu1} and \cite{eu2} on the reticulation of a residuated lattice. All the definitions and properties from the previous two articles that we need in the sequel can be found in Section \ref{preliminaries}, together with other known results and a few simple new ones that will also be necessary in the following sections.

In \cite{eu2} we proved certain preservation properties of the reticulation functor for residuated lattices, ${\cal{L}}$. In Section \ref{preservation} we continue this study, proving that ${\cal{L}}$ preserves subalgebras and finite direct products, does not preserve quotients (but there exists a surjective morphism from the quotient of the reticulation to the reticulation of the quotient), preserves inductive limits and Boolean powers. Our goal is to use such preservation results for a transfer of properties between residuated lattices and bounded distributive lattices through the reticulation.

We begin the research of such properties that can be transferred through the reticulation in Section \ref{stonealgebras}, which is focused on Stone and strongly Stone lattices and residuated lattices. Here, we prove that a residuated lattice is Stone (respectively strongly Stone, respectively $m$-Stone) iff its reticulation is Stone (respectively strongly Stone, respectively $m$-Stone). We also show that the Boolean algebra of the co-annihilator filters of a residuated lattice and that of its reticulation are isomorphic. Through the reticulation, we transfer a known characterization of $m$-Stone bounded distributive lattices to $m$-Stone residuated lattices. We conclude by two remarks related to a characterization of Stone pseudocomplemented distributive lattices that is not valid for residuated lattices.

In future articles we will continue our research on the transfer of pro\-per\-ties between the category of bounded distributive lattices and that of residuated lattices through the reticulation functor. This transfer of pro\-per\-ties between different categories is the very purpose of the reticulation.

\section{Preliminaries}
\label{preliminaries}

\begin{definition}
Let $L$ be a distributive lattice with $0$. An element $l$ of $L$ is said to be {\em pseudocomplemented} iff there exists a maximal element $m$ of $L$ which satisfies: $l\wedge m=0$; such an element $m$ is denoted $l^{*}$ and called the {\em pseudocomplement of $l$}. $L$ is said to be {\em pseudocomplemented} iff all its elements are pseudocomplemented.
\end{definition}

\begin{definition}
Let $L$ be a lattice. A nonempty subset $F$ of $L$ is called a {\em filter of $L$} iff it satisfies the following conditions:

\noindent (i) for all $l,m\in F$, $l\wedge m\in F$;

\noindent (ii) for all $l\in F$ and all $m\in L$, if $l\leq m$ then $m\in F$.

The set of all filters of $L$ is denoted ${\cal{F}}(L)$.

\label{`filterl`}
\end{definition}

Let $L$ be a lattice and $F$ a filter of $L$. For all $l,m\in L$, we denote $l\equiv m({\rm mod}\ F)$ and say that {\em $l$ and $m$ are congruent modulo $F$} iff there exists an element $e\in F$ such that $l\wedge e=m\wedge e$. Obviously, $\equiv ({\rm mod}\ F)$ is a congruence relation on $L$. The quotient lattice with respect to the congruence relation $\equiv ({\rm mod}\ F)$ is denoted $L/F$ and its elements are denoted $l/F$, $l\in L$.

\begin{definition}
A {\em residuated lattice} is an algebraic structure $(A,\vee ,\wedge ,\odot ,\rightarrow ,0,1)$, with the first 4 operations binary and the last two constant, such that $(A,\vee ,\wedge ,0,1)$ is a bounded lattice, $(A,\odot ,1)$ is a commutative monoid and the following property, called {\em residuation}, is satisfied: for all $a,b,c\in A$, $a\leq b\rightarrow c\Leftrightarrow a\odot b\leq c$, where $\leq $ is the partial order of the lattice $(A,\vee ,\wedge ,0,1)$.
\label{`rl`}
\end{definition}

For any residuated lattice $A$ and any $a,b\in A$, we denote $a\leftrightarrow b=(a\rightarrow b)\wedge (b\rightarrow a)$ and $\neg \, a=a\rightarrow 0$.

\begin{definition}
Let $A$ be a residuated lattice. A nonempty subset $F$ of $A$ is called a {\em filter of $A$} iff it satisfies the following conditions:

\noindent (i) for all $a,b\in F$, $a\odot b\in F$;

\noindent (ii) for all $a\in F$ and all $b\in A$, if $a\leq b$ then $b\in F$.

The set of all filters of $A$ is denoted ${\cal{F}}(A)$.

\label{`filtera`}
\end{definition}

\begin{definition}
Let $A$ be a lattice (residuated lattice), $X\subseteq A$ and $a\in A$. The least filter  of $A$ that includes $X$ (that is: the intersection of all filters of $A$ that include $X$) is called {\em the filter of $A$ generated by $X$} and is denoted by $<X>$. The filter of $A$ generated by $\{a\}$ is denoted by $<a>$ and is called {\em the principal filter of $A$ generated by $a$}. For lattices, the notations mentioned above can be replaced by $[X)$ and respectively $[a)$.
\label{`filtrulgenerat`}
\end{definition}

\begin{lemma}{\rm \cite[Lemma 2.1]{eu1}}\\ Let $A$ be a residuated lattice and $a\in A$. Then $<a>=\{b\in A|(\exists \, n\in {\rm I\!N}^{*})\, a^{n}\leq b\}$.
\label{`principal`}
\end{lemma}

\begin{notation}
Let $A$ be a lattice (residuated lattice). For all filters $F$, $G$ of $A$, we denote $<F\cup G>$ by $F\vee G$. More generally, for any family $\{F_{t}|t\in T\}$ of filters of $A$, we denote $<\displaystyle \bigcup _{t\in T}F_{t}>$ by $\displaystyle \bigvee _{t\in T}F_{t}$.
\label{`veefamfiltre`}
\end{notation}

\begin{proposition}{\rm \cite[Proposition 3.5]{eu1}}\\ Let $A$ be a bounded distributive lattice or a residuated lattice. Then $({\cal{F}}(A),$\linebreak $\vee ,\cap ,\{1\},A)$ is a bounded distributive lattice, whose order relation is $\subseteq $.
\end{proposition}

\begin{proposition}{\rm \cite[Proposition 4.3]{eu1}}\\ Let $A$ be a residuated lattice. Then, for all $a,b\in A$, $<a>\cap <b>=<a\vee b>$.
\label{`sisau`}
\end{proposition}

Let $A$ be a residuated lattice and $F$ a filter of $A$. For all $a,b\in A$, we denote $a\equiv b({\rm mod}\ F)$ and say that {\em $a$ and $b$ are congruent modulo $F$} iff $a\leftrightarrow b\in F$. Obviously, $\equiv ({\rm mod}\ F)$ is a congruence relation on $A$. The quotient residuated lattice with respect to the congruence relation $\equiv ({\rm mod}\ F)$ is denoted $A/F$ and its elements are denoted $a/F$, $a\in A$.

\begin{remark}{\rm \cite{pic}}\\ Let $A$ be a residuated lattice and $a,b,c,d\in A$. Then:

\noindent (i) $a\odot (b\vee c)=(a\odot b)\vee (a\odot c)$;

\noindent (ii) if $a\vee b=1$, then $a\odot b=a\wedge b$;

\noindent (iii) if $a\leq b$ and $c\leq d$, then $a\odot c\leq b\odot d$; 

\noindent (iv) $a\leq b$ iff $a\rightarrow b=1$.
\label{`calcul`}
\end{remark}

If $A$ is a residuated lattice, we call the {\em Boolean center of $A$} the set of the complemented elements of $A$, which we denote by $B(A)$. It is known that this subset of $A$ is a Boolean algebra with the operations induced by those of $A$.

In \cite{eu1} we gave the following definition of the reticulation of a residuated lattice.

\begin{definition}{\rm \cite{eu1}}\\ Let $A$ be a residuated lattice. A {\em reticulation of $A$} is a pair $(L,\lambda )$, where $L$ is a bounded distributive lattice and $\lambda :A\rightarrow L$ is a function that satisfies conditions 1)-5) below:
 
\noindent 1) for all $a,b\in A$, $\lambda (a\odot b)=\lambda (a)\wedge \lambda (b)$;

\noindent 2) for all $a,b\in A$, $\lambda (a\vee b)=\lambda (a)\vee \lambda (b)$;

\noindent 3) $\lambda (0)=0$; $\lambda (1)=1$;

\noindent 4) $\lambda $ is surjective;

\noindent 5) for all $a,b\in A$, $\lambda (a)\leq \lambda (b)$ iff $(\exists \, n\in {\rm I\! N}^{*})\, a^{n}\leq b$.
\label{reticulatia}
\end{definition}

\begin{lemma}{\rm \cite[Lemma 3.1]{eu1}}\\ If $A$ is a residuated lattice and $L$ is a bounded distributive lattice, then a function $\lambda :A\rightarrow L$ that verifies conditions 1)-3) also satisfies:

\noindent a) $\lambda $ is order-preserving;

\noindent b) for all $a,b\in A$, $\lambda (a\wedge b)=\lambda (a)\wedge \lambda (b)$;

\noindent c) for all $a\in A$ and all $n\in {\rm I\! N}^{*}$, $\lambda (a^{n})=\lambda (a)$.
\label{`abc`}
\end{lemma}

We shall use the notations of the conditions 1)-5) and of the properties a)-c) in the following sections also.

Until mentioned otherwise, let $A$ be a residuated lattice and $(L,\lambda )$ a reticulation of $A$.

\begin{lemma}{\rm \cite[Lemma 3.4]{eu1}}\\ For any filter $F$ of $A$ and any $a\in A$, we have: $\lambda (a)\in \lambda (F)$ iff $a\in F$.
\label{`apartine2`}
\end{lemma}

\begin{lemma}
For any filters $F$, $G$ of $A$, we have: $\lambda (F)=\lambda (G)$ iff $F=G$.
\label{lambdaegal}
\end{lemma}

\begin{proof}
By double inclusion, using Lemma \ref{`apartine2`}.\end{proof}

\begin{remark}
For all $a\in A$, $\lambda (<a>)=<\lambda (a)>$.
\label{lambdagen}
\end{remark}

\begin{proof}
This is an obvious consequence of Lemmas \ref{`principal`} and \ref{lambdaegal} and conditions 4) and 5).\end{proof}

Lemma \ref{lambdaegal} could have been obtained as a corollary of the following proposition.

\begin{proposition}{\rm \cite[Proposition 3.6]{eu1}}\\ The function ${\cal{F}}(A)\rightarrow {\cal{F}}(L)$, $F\rightarrow \lambda (F)$, is well defined and it is a bounded lattice isomorphism.
\label{`izomf`}
\end{proposition}

\begin{lemma}{\rm \cite[Lemma 3.3]{eu1}}\\ For all $a,b\in A$, $\lambda (a)=\lambda (b)$ iff $<a>=<b>$.
\label{`egalfiltre`}
\end{lemma}

Notice that Lemma \ref{`egalfiltre`} could easily have been obtained as a consequence of Lemma \ref{lambdaegal} and Remark \ref{lambdagen}.


The following theorem states the existence and uniqueness of the re\-ti\-cu\-la\-tion for any residuated lattice.

\begin{theorem}{\rm \cite{eu1}}\\ Let $A$ be a residuated lattice. Then there exists a reticulation of $A$. Let $(L_{1},\lambda _{1})$, $(L_{2},\lambda _{2})$ be two reticulations of $A$. Then there exists an isomorphism of bounded lattices $f:L_{1}\rightarrow L_{2}$ such that $f\circ \lambda _{1}=\lambda _{2}$.
\label{`unicitatea`}
\end{theorem}

We denote by ${\cal{RL}}$ the category of residuated lattices and by ${\cal{D}}01$ the category of bounded distributive lattices.

In \cite{eu1} and \cite{eu2}, we defined {\em the reticulation functor} ${\cal{L}}:{\cal{RL}}\rightarrow {\cal{D}}01$. If $A$ is a residuated lattice and $(L(A),\lambda _{A})$ is its reticulation, then ${\cal{L}}(A)=L(A)$. If $B$ is another residuated lattice, $(L(B),\lambda _{B})$ is its reticulation and $f:A\rightarrow B$ is a morphism of residuated lattices, then ${\cal{L}}(f):{\cal{L}}(A)=L(A)\rightarrow {\cal{L}}(B)=L(B)$, for all $a\in A$, ${\cal{L}}(f)(\lambda _{A}(a))=\lambda _{B}(f(a))$.

In \cite{eu1} we constructed the reticulation of a residuated lattice in two different ways. Here is the second construction of the reticulation that we showed in that article. Let $A$ be a residuated lattice and let us denote by ${\cal{PF}}(A)$ the set of principal filters of $A$. Also, we denote by $\lambda :A\rightarrow {\cal{PF}}(A)$ the function given by: for all $a\in A$, $\lambda (a)=<a>$. Then, according to Theorem 4.2 in \cite{eu1}, $(({\cal{PF}}(A),\cap ,\vee ,A,\{1\}),\lambda )$ is a reticulation of $A$. The definition of functor ${\cal{L}}$ using this construction is obvious: ${\cal{L}}(A)={\cal{PF}}(A)$ and, if $f:A\rightarrow B$ is a morphism of residuated lattices, then, for all $a\in A$, ${\cal{L}}(f)(<a>)=<f(a)>$.

For the definitions related to the inductive limit, that we present below, we are using the terminology of \cite{bus}.

A partially ordered set $(I,\leq )$ is called a {\em directed set} iff, for any $i,j\in I$, there exists an element $k\in I$ such that $i\leq k$ and $j\leq k$.

\begin{definition}
Let $(I,\leq )$ be a directed set and $\cal{C}$ a category. We call {\em inductive system} of objects in $\cal{C}$ with respect to the directed index set $I$ a pair $((A_{i})_{i\in I},(\phi _{ij})_{\stackrel{\scriptstyle i,j\in I}{\scriptstyle i\leq j}})$ with $(A_{i})_{i\in I}$ a family of objects of $\cal{C}$ and, for all $i,j\in I$ with $i\leq j$, $\phi _{ij}:A_{i}\rightarrow A_{j}$ a morphism in ${\cal{C}}$, such that:

\noindent (i) for every $i\in I$, $\phi _{i\, i}=1_{A_{i}}$;

\noindent (ii) for any $i,j,k\in I$ with $i\leq j\leq k$, $\phi _{jk}\circ \phi _{ij}=\phi _{ik}$.

If there is no danger of confusion, an inductive system like above will be denoted $(A_{i},\phi _{ij})$.
\end{definition}

\begin{definition}
Let $(A_{i},\phi _{ij})$ be an inductive system of objects in a ca\-te\-go\-ry ${\cal{C}}$ relative to a directed index set $I$. A pair $(A,(\phi _{i})_{i\in I})$, with $A$ an object in ${\cal{C}}$ and, for all $i\in I$, $\phi _{i}:A_{i}\rightarrow A$ a morphism in ${\cal{C}}$, is called {\em inductive limit} of the inductive system $(A_{i},\phi _{ij})$ iff:

\noindent (i) for every $i,j\in I$ with $i\leq j$, $\phi _{j}\circ \phi _{ij}=\phi _{i}$;

\begin{center}
\begin{picture}(60,60)(0,0)
\put(7,37){$A_{i}$}
\put(20,40){\vector(1,0){20}}
\put(42,37){$A_{j}$}
\put(22,44){$\phi _{ij}$}
\put(45,35){\vector(0,-1){20}}
\put(47,22){$\phi _{j}$}
\put(42,5){$A$}
\put(18,35){\vector(1,-1){22}}
\put(16,22){$\phi _{i}$}

\end{picture}
\end{center}

\noindent (ii) for any object $B$ of ${\cal{C}}$ and any family $(f_{i})_{i\in I}$ of morphisms in ${\cal{C}}$ such that, for all $i\in I$, $f_{i}:A_{i}\rightarrow B$ and, for all $i,j\in I$ with $i\leq j$, $f_{j}\circ \phi _{ij}=f_{i}$, there is a unique morphism $f:A\rightarrow B$ in ${\cal{C}}$ such that, for every $i\in I$, $f\circ \phi _{i}=f_{i}$.

\begin{center}
\begin{picture}(60,60)(0,0)
\put(7,37){$A_{i}$}
\put(20,40){\vector(1,0){20}}
\put(42,37){$A$}
\put(22,44){$\phi _{i}$}

\put(45,35){\vector(0,-1){20}}
\put(47,22){$f$}
\put(42,5){$B$}
\put(18,35){\vector(1,-1){22}}
\put(16,22){$f_{i}$}
\end{picture}
\end{center}
\end{definition}

It is immediate that the inductive limit of a given inductive system is unique up to an isomorphism, that is, if $(A,(\phi _{i})_{i\in I})$ and $(B,(\psi _{i})_{i\in I})$ are two inductive limits of the same inductive system, then there exists a unique isomorphism $f:A\rightarrow B$ such that, for every $i\in I$, $f\circ \phi _{i}=\psi _{i}$.

We say that a category ${\cal{C}}$ is a {\em category with inductive limits} iff every inductive system in ${\cal{C}}$ has an inductive limit. The category of sets, the category of residuated lattices and the category of bounded distributive lattices are categories with inductive limits.

In the following, we shall present a construction for the inductive limit in the category of residuated lattices. As we believe that this construction is known, we shall not give any proofs here. See also \cite{bus}.

Let $(A_{i},\phi _{ij})$ be an inductive system in ${\cal{RL}}$. We denote by $\displaystyle \coprod _{i\in I}A_{i}$ the disjoint union of the family $(A_{i})_{i\in I}$. Let us consider the following relation on $\displaystyle \coprod _{i\in I}A_{i}$: for all $i,j\in I$, all $a\in A_{i}$ and all $b\in A_{j}$, $a\sim b$ iff there exists $k\in I$ such that $i\leq k$, $j\leq k$ and $\phi _{ik}(a)=\phi _{jk}(b)$. It is immediate that $\sim $ is an equivalence relation on $\displaystyle \coprod _{i\in I}A_{i}$. We denote by $A$ the quotient set $\left( \displaystyle \coprod _{i\in I}A_{i}\right) /\sim$ and by $[a]$ the equivalence class of an element $a\in \displaystyle \coprod _{i\in I}A_{i}$. For any $i\in I$, let $\phi _{i}:A_{i}\rightarrow A$, for all $a\in A_{i}$, $\phi _{i}(a)=[a]$.

Let us define residuated lattice operations on $A$. We define $0=[0]$ and $1=[1]$. Obviously, this definition does not depend on the residuated lattice $A_{i}$ the 0 and the 1 are taken from. Let $[a],[b]\in A$. Let $i,j\in I$ such that $a\in A_{i}$ and $b\in A_{j}$. Then, by the definition of the directed set, there exists $k\in I$ such that $i\leq k$ and $j\leq k$. We define $[a]\vee [b]=[\phi _{ik}(a)\vee \phi _{jk}(b)]$ and $[a]\wedge [b]=[\phi _{ik}(a)\wedge \phi _{jk}(b)]$. The same for $\odot $ and $\rightarrow $. Here is the definition of the partial order relation: for all $a,b\in A$ with $a\in A_{i}$ and $b\in A_{j}$ for some $i,j\in I$, we define: $[a]\leq [b]$ iff there exists $k\in I$ such that $i\leq k$, $j\leq k$ and $\phi _{ik}(a)\leq \phi _{jk}(b)$.

Then $(A,(\phi _{i})_{i\in I})$ is an inductive limit of the inductive system $(A_{i},\phi _{ij})$ in the category $\cal{RL}$.

A similar construction can be done for inductive limits in the category ${\cal{D}}01$.

In the following, let $(A,F)$ be a universal algebra (we will use the definitions and notations from \cite{bur} here) and $B$ a Boolean algebra. We denote by $A[B]$ the set of the functions $X:A\rightarrow B$ which verify: $X(A)$ is finite, $\displaystyle \bigvee _{a\in A}X(a)=1$ and, for all $a,b\in A$, if $a\neq b$ then $X(a)\wedge X(b)=0$. $A[B]$ is an algebra of the same type as $A$, with the operations defined this way: if $f$ is a $n$-ary operation in $F$ and $X_{1},\ldots ,X_{n}\in A[B]$, then $f(X_{1},\ldots ,X_{n})\in A[B]$, for all $a\in A$, $f(X_{1},\ldots ,X_{n})(a)=\bigvee \{X_{1}(a_{1})\wedge \ldots \wedge X_{n}(a_{n})|a_{1},\ldots ,a_{n}\in A,f(a_{1},\ldots ,a_{n})=a\}$. We call $A[B]$, with these operations, a {\em Boolean power of $A$}.

Now let $P(B)$ be the set of the finite partitions of $B$, that is $\displaystyle P(B)=\{\{x_{1},\ldots ,x_{n}\}|n\in {\rm I\! N}^{*},x_{1},\ldots ,x_{n}\in B\setminus \{0\},\bigvee _{i=1}^{n}x_{i}=1, (\forall i,j\in \overline{1,n})i\neq j\Rightarrow x_{i}\wedge x_{j}=0\}$. We define the partial order $\leq $ on $P(B)$ by: for all $p,q\in P(B)$, $p\leq q$ iff $q$ is a refinement of $p$, that is: $p=\{x_{1},\ldots ,x_{n}\}$ and $q=\{y_{ij}|i\in \overline{1,n},(\forall i\in \overline{1,n})j\in \overline{1,k_{i}}\}$, where $n,k_{1},\ldots ,k_{n}\in {\rm I\! N}^{*}$ and, for all $i\in \overline{1,n}$, $\displaystyle \bigvee _{j=1}^{k_{i}}y_{ij}=x_{i}$. For all $p,q\in P(B)$ with $p\leq q$, we define $k_{pq}:q\rightarrow p$, for all $a\in q$ and $b\in p$, $k_{pq}(a)=b$ iff $a\leq b$ (with the notations above for the elements of $p$ and those of $q$, for all $i\in \overline{1,n}$ and all $j\in \overline{1,k_{i}}$, $k_{pq}(y_{ij})=x_{i}$). The fact that the functions $k_{pq}$ are well defined is obvious (if, for an $a\in q$, there exist $b_{1},b_{2}\in p$, $b_{1}\neq b_{2}$ and $a\leq b_{1}$, $a\leq b_{2}$, then $a\leq b_{1}\wedge b_{2}=0$, so $a=0$, which is a contradiction to the definition of $P(B)$).

For every $p\in P(B)$, we define $A^{p}=\{X|X:p\rightarrow A\}$, organized as a universal algebra of the type of $A$ like this: if $f$ is a $n$-ary operation in $F$ and $X_{1},\ldots ,X_{n}\in A^{p}$, then $f(X_{1},\ldots ,X_{n})\in A^{p}$, for all $a\in A$, $f(X_{1},\ldots ,X_{n})(a)=\bigvee \{X_{1}(a_{1})\wedge \ldots \wedge X_{n}(a_{n})|a_{1},\ldots ,a_{n}\in A,f(a_{1},\ldots ,a_{n})=a\}$. For all $p,q\in P(B)$ such that $p\leq q$, $k_{pq}$ induces a morphism of universal algebras of the type of $A$, $f_{pq}:A^{p}\rightarrow A^{q}$, for all $X\in A^{p}$ and $a\in q$, $f_{pq}(X)(a)=X(k_{pq}(a))$. It is easily seen that $((A^{p})_{p\in P(B)},(f_{pq})_{\stackrel{\scriptstyle p,q\in P(B)}{\scriptstyle p\leq q}})$ is an inductive system. We shall denote by $\lim _{B}(A)$ its inductive limit.

\begin{theorem}{\rm \cite[Theorem 3]{ash}}\\ With the notations above, $\lim _{B}(A)$ exists and it is isomorphic to $A[B]$ as universal algebras of type $F$.
\label{tash}
\end{theorem}

Until mentioned otherwise, let $A$ be a bounded distributive lattice or a residuated lattice; the definitions we are about to give are valid for both types of structures. For any non-empty subset $X$ of $A$, the {\em co-annihilator of $X$} is the set $X^{\top }=\{a\in A|(\forall x\in X)a\vee x=1\}$. In the case when $X$ consists of a single element $x$, we denote the co-annihilator of $X$ by $x^{\top }$ and call it the {\em co-annihilator of $x$}.

\begin{proposition}
For any $X\subseteq A$, $X^{\top }$ is a filter of $A$.
\label{topfiltru}
\end{proposition}

\begin{proof}
This result can be found in \cite[Proposition 4.38]{din2} for BL-algebras. The proof there is valid also for bounded distributive lattices and for re\-si\-du\-a\-ted lattices.\end{proof}

\begin{definition}
$A$ is said to be {\em Stone} (respectively {\em strongly Stone}) iff, for all $a\in A$ (respectively all $X\subseteq A$), there exists an element $e\in B(A)$ such that $a^{\top }=<e>$ (respectively $X^{\top }=<e>$).
\label{stone}
\end{definition}

Obviously, any complete Stone lattice (residuated lattice) is strongly Stone, as is shown by Proposition \ref{`sisau`} and the fact that, with the notations in the previous definition, $\displaystyle X^{\top }=\bigcap _{x\in X}x^{\top }$.

We have chosen the previous definition of Stone residuated lattices over the definition from \cite{rcig} for a reason that is explained by Remark \ref{nucdefstone}.

For any bounded distributive lattice or residuated lattice $A$, we shall denote ${\rm Co-Ann}(A)=\{X^{\top }|X\subseteq A\}$ and, for all $F,G\in {\rm Co-Ann}(A)$, we shall denote $F\vee ^{\top }G=(F^{\top }\cap G^{\top })^{\top }$.

\begin{proposition}
Let $A$ be a bounded distributive lattice or a residuated lattice. Then $({\rm Co-Ann}(A),\vee ^{\top },\cap ,^{\top },\{1\},A)$ is a complete Boolean algebra.
\end{proposition}

\begin{proof}
This result can be found in \cite{leo} for BL-algebras. Its proof is also valid for bounded distributive lattices and residuated lattices.\end{proof}

\begin{proposition}

Let $A$ be a residuated lattice and $({\cal{L}}(A),\lambda )$ the re\-ti\-cu\-la\-tion of $A$. Then $\lambda :B(A)\rightarrow B({\cal{L}}(A))$ is an isomorphism of Boolean algebras.
\label{izomb}
\end{proposition}

\begin{proof}
This result can be found in \cite{leo} for BL-algebras. The proof there is also valid for the more general case of residuated lattices.\end{proof}

\begin{definition}
Let $m$ be an infinite cardinal. An $m$-complete lattice is a lattice $L$ with the property that any subset $X$ of $L$ with $|X|\leq m$ has an infimum and a supremum in $L$.
\end{definition}

\begin{theorem}{\rm \cite[Theorem 1]{dav}}\\ Let $L$ be a bounded distributive lattice and $m$ an infinite cardinal. Then the following are equivalent:

\noindent (i) for each subset $X$ of $L$ with $|X|\leq m$, there exists an element $e\in B(L)$ such that $X^{\top }=<e>$;

\noindent (ii) $L$ is a Stone lattice and $B(L)$ is an $m$-complete Boolean algebra;

\noindent (iii) $L^{\top \top }=\{(l^{\top })^{\top }|l\in L\}$ is an $m$-complete Boolean sublattice of ${\cal{F}}(L)$;

\noindent (iv) for all $l,p\in L$, $(l\wedge p)^{\top }=l^{\top }\vee p^{\top }$ and, for each subset $X$ of $L$ with $|X|\leq m$, there exists an element $x\in L$ such that $X^{\top \top}=x^{\top }$;

\noindent (v) for each subset $X$ of $L$ with $|X|\leq m$, $X^{\top }\vee X^{\top \top }=L$.
\label{caractlstone}
\end{theorem}

A bounded distributive lattice will be called an {\em $m$-Stone lattice} iff the conditions of Theorem \ref{caractlstone} hold for it.

\section{Further Preservation Properties of the Reticulation Functor}
\label{preservation}

\hspace*{11pt} In this section we continue the study we began in \cite{eu2} on preservation properties of ${\cal{L}}$.

\begin{proposition}
${\cal{L}}$ preserves subalgebras. Namely, if $A$ and $B$ are residuated lattices such that $B$ is a subalgebra of $A$ and $({\cal{L}}(A),\lambda )$ is a reticulation of $A$, then $(\lambda (B),\lambda \mid _{B})$ is a reticulation of $B$.
\end{proposition}
\begin{proof}
From properties 2), 3) and b) it follows that $\lambda (B)$ is a bounded lattice. Properties 4) and 1) and Remark \ref{`calcul`}, (i), ensure us that it is also distributive.

The fact that $\lambda $ verifies properties 1), 2), 3), 5) implies that $\lambda \mid _{B}$ satisfies these properties. Obviously, $\lambda \mid _{B}:B\rightarrow \lambda (B)$ satisfies condition 4).\end{proof}

\begin{proposition}
${\cal{L}}$ preserves finite direct products.
\end{proposition}
\begin{proof}
Let $A_{1},A_{2},\ldots A_{n}$ be residuated lattices and $\displaystyle A=\prod _{i=1}^{n}A_{i}$. Let $({\cal{L}}(A_{i}),$\linebreak $\lambda _{i})$ be a reticulation of $A_{i}$, for each $i\in \overline{1,n}$, and $\displaystyle \lambda :A\rightarrow \prod _{i=1}^{n}{\cal{L}}(A_{i})$, for all $(a_{1},\ldots ,a_{n})\in A$, $\lambda (a_{1},\ldots ,a_{n})=(\lambda _{1}(a_{1}),\ldots ,\lambda _{n}(a_{n}))$.

We shall prove that $\displaystyle (\prod _{i=1}^{n}{\cal{L}}(A_{i}),\lambda )$ is a reticulation of $A$.

The fact that $\lambda _{1},\ldots ,\lambda _{n}$ satisfy conditions 1)-4) implies that $\lambda $ satisfies conditions 1)-4).

Let us now prove that $\lambda $ satisfies condition 5). Let $a=(a_{1},\ldots ,a_{n}),b=(b_{1},\ldots ,b_{n})\in A$. Assume that there exists $m\in {\rm I\! N}^{*}$ such that $a^{m}\leq b$, which is equivalent to: for all $i\in \overline{1,n}$, $a_{i}^{m}\leq b_{i}$. This implies that, for all $i\in \overline{1,n}$, $\lambda _{i}(a_{i})\leq \lambda _{i}(b_{i})$, that is: $\lambda (a)\leq \lambda (b)$. Conversely, suppose that $\lambda (a)\leq \lambda (b)$, that is: for all $i\in \overline{1,n}$, $\lambda _{i}(a_{i})\leq \lambda _{i}(b_{i})$, which is equivalent to: for all $i\in \overline{1,n}$, there exists $m_{i}\in {\rm I\! N}^{*}$ such that $a_{i}^{m_{i}}\leq b_{i}$. If we denote $m=\max \{m_{i}|i\in \overline{1,n}\}$, we get: for all $i\in \overline{1,n}$, $a_{i}^{m}\leq b_{i}$, that is: $a^{m}\leq b$ (See Remark \ref{`calcul`}, (iii)).\end{proof}

\begin{proposition}
${\cal{L}}$ does not preserve quotients.
\end{proposition}

\begin{proof}
Here is an example of residuated lattice from \cite{kow}: $A=\{0,a,b,c,d,1\}$, with the structure described below.

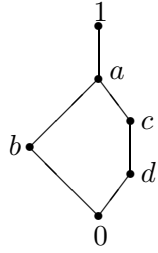
\begin{center}
\begin{picture}(60,100)(0,0)

\put(37,11){\circle*{3}}
\put(35,0){$0$}

\put(37,11){\line(3,4){12}}
\put(49,27){\circle*{3}}
\put(53,24){$d$}
\put(49,27){\line(0,1){20}}
\put(49,47){\circle*{3}}

\put(53,44){$c$}

\put(49,47){\line(-3,4){12}}
\put(37,63){\circle*{3}}
\put(41,63){$a$}

\put(37,11){\line(-1,1){26}}
\put(11,37){\circle*{3}}
\put(3,34){$b$}
\put(11,37){\line(1,1){26}}
\put(37,63){\line(0,1){20}}

\put(37,83){\circle*{3}}
\put(35,85){$1$}
\end{picture}
\end{center}

\begin{center}
\begin{tabular}{cc}
\begin{tabular}{c|cccccc}
$\rightarrow $ & $0$ & $a$ & $b$ & $c$ & $d$ & $1$ \\ \hline
$0$ & $1$ & $1$ & $1$ & $1$ & $1$ & $1$ \\
$a$ & $0$ & $1$ & $b$ & $c$ & $c$ & $1$ \\
$b$ & $c$ & $1$ & $1$ & $c$ & $c$ & $1$ \\
$c$ & $b$ & $1$ & $b$ & $1$ & $a$ & $1$ \\
$d$ & $b$ & $1$ & $b$ & $1$ & $1$ & $1$ \\
$1$ & $0$ & $a$ & $b$ & $c$ & $d$ & $1$
\end{tabular}
& \hspace*{11pt}

\begin{tabular}{c|cccccc}
$\odot $ & $0$ & $a$ & $b$ & $c$ & $d$ & $1$ \\ \hline
$0$ & $0$ & $0$ & $0$ & $0$ & $0$ & $0$ \\
$a$ & $0$ & $a$ & $b$ & $d$ & $d$ & $a$ \\
$b$ & $0$ & $b$ & $b$ & $0$ & $0$ & $b$ \\
$c$ & $0$ & $d$ & $0$ & $d$ & $d$ & $c$ \\
$d$ & $0$ & $d$ & $0$ & $d$ & $d$ & $d$ \\
$1$ & $0$ & $a$ & $b$ & $c$ & $d$ & $1$
\end{tabular}
\end{tabular}
\end{center}

We choose the filter $F=<a>=\{a,1\}$ and show that ${\cal{L}}(A/F)$ and ${\cal{L}}(A)/\lambda (F)$ are not isomorphic.

For determining the image of the reticulation functor we shall use the second construction of the reticulation from \cite{eu1}, that we reminded in Section \ref{preliminaries}.

$<0>=A$, $<a>=\{a,1\}$, $<b>=\{b,a,1\}$, $<c>=<d>=\{c,d,a,1\}$, $<1>=\{1\}$, so ${\cal{L}}(A)=\{<0>,<a>,<b>,<c>,<1>\}$, with the following lattice structure:

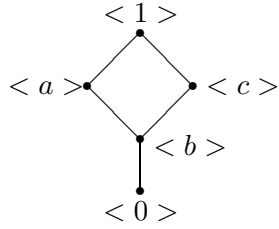
\begin{center}
\begin{picture}(100,90)(0,0)

\put(50,11){\circle*{3}}
\put(36,0){$<0>$}
\put(50,11){\line(0,1){20}}

\put(50,31){\circle*{3}}
\put(55,25){$<b>$}
\put(50,31){\line(1,1){20}}
\put(70,51){\circle*{3}}
\put(75,48){$<c>$}

\put(50,31){\line(-1,1){20}}
\put(30,51){\circle*{3}}
\put(0,48){$<a>$}
\put(50,71){\line(1,-1){20}}
\put(50,71){\line(-1,-1){20}}
\put(50,71){\circle*{3}}
\put(36,75){$<1>$}
\end{picture}
\end{center}

$\lambda (F)=\{<a>,<1>\}$. ${\cal{L}}(A)/\lambda (F)=\{<0>/\lambda (F),<a>/\lambda (F),<b>/\lambda (F),<c>/\lambda (F),<1>/\lambda (F)\}$. For all $x\in A$, $<x>/\lambda (F)=\{<y>\in {\cal{L}}(A)|(\exists <e>\in \lambda (F))<x>\wedge <e>=<y>\wedge <e>\}$. $<0>/\lambda (F)=\{<0>\}$, $<b>/\lambda (F)=\{<b>,<c>\}=<c>/\lambda (F)$, $<a>/\lambda (F)=<1>/\lambda (F)=\lambda (F)$. So ${\cal{L}}(A)/\lambda (F)=\{<0>/\lambda (F),<b>/\lambda (F),<1>/\lambda (F)\}$, which is the bounded distributive lattice with three elements.

The table of the operation $\leftrightarrow $ on $A$ shows that: $0/F=\{0\}$, $a/F=1/F=F=\{a,1\}$, $b/F=\{b\}$, $c/F=\{c,d\}=d/F$, hence $A/F=\{0/F,b/F,c/F,1/F\}$ and ${\cal{L}}(A/F)=\{<0/F>,<b/F>,<c/F>,<1/F>\}$. For all $x\in A$, $<x/F>=\{y/F\in A/F|(\exists n\in {\rm I\! N}^{*})(x/F)^{n}\leq y/F\}=\{y/F\in A/F|(\exists n\in {\rm I\! N}^{*})x^{n}/F\leq y/F\}=\{y/F\in A/F|(\exists n\in {\rm I\! N}^{*})x^{n}\rightarrow y\in F\}$. We get: $<0/F>=A/F$, $<b/F>=\{a/F,b/F,1/F\}$, $<c/F>=\{a/F,c/F,1/F\}$, $<1/F>=\{1/F\}$. Therefore ${\cal{L}}(A/F)$ has four distinct elements.

So ${\cal{L}}(A/F)$ and ${\cal{L}}(A)/\lambda (F)$ are not isomorphic, as their cardinalities are different.\end{proof}

\begin{remark}
Let $A$ be a residuated lattice, $F$ a filter of $A$ and $({\cal{L}}(A),\lambda )$ the reticulation of $A$. Then there exists a surjective bounded lattice morphism from ${\cal{L}}(A)/\lambda (F)$ to ${\cal{L}}(A/F)$.
\end{remark}

\begin{proof}
Let $({\cal{L}}(A/F),\lambda _{1})$ be the reticulation of $A/F$. Let $h:{\cal{L}}(A)/\lambda (F)\rightarrow {\cal{L}}(A/F)$, for all $a\in A$, $h(\lambda (a)/\lambda (F))=\lambda _{1}(a/F)$. The surjectivity of $\lambda $ implies that $h$ is completely defined.

Let $a,b\in A$. By Lemma \ref{`apartine2`}, $\lambda (a)/\lambda (F)=\lambda (b)/\lambda (F)$ iff $\lambda (a)\leftrightarrow \lambda (b)\in \lambda (F)$ iff $\lambda (a\leftrightarrow b)\in \lambda (F)$ iff $a\leftrightarrow b \in F$ iff $a/F=b/F$, which implies $\lambda _{1}(a/F)=\lambda _{1}(b/F)$. Hence the function $h$ is well defined. Obviously, the converse implication is not necessarily satisfied, so $h$ is not always injective.

Since $\lambda _{1}$ is surjective, we have that $h$ is surjective.

The fact that $\lambda $ and $\lambda _{1}$ satisfy conditions 1), 2) and 3) implies that $h$ is a bounded lattice morphism.\end{proof}


\begin{proposition}
${\cal{L}}$ preserves inductive limits.
\label{limind}
\end{proposition}

\begin{proof}
Let $((A_{i})_{i\in I},(\phi _{ij})_{\stackrel{\scriptstyle i,j\in I}{\scriptstyle i\leq j}})$ be an inductive system of residuated lattices and $(A,(\phi _{i})_{i\in I})$ its inductive limit, constructed like in Section \ref{preliminaries}. For all $i\in I$, let $({\cal{L}}(A_{i}),\lambda _{i})$ be the reticulation of $A_{i}$ and $({\cal{L}}(A),\lambda )$ the reticulation of $A$. Then, obviously, $(({\cal{L}}(A_{i}))_{i\in I},({\cal{L}}(\phi _{ij}))_{\stackrel{\scriptstyle i,j\in I}{\scriptstyle i\leq j}})$ is an inductive system of bounded distributive lattices. We shall prove that $({\cal{L}}(A),({\cal{L}}(\phi _{i}))_{i\in I})$ is its inductive limit.

For all $i\leq j$, we have $\phi _{j}\circ \phi _{ij}=\phi _{i}$, thus ${\cal{L}}(\phi _{j})\circ {\cal{L}}(\phi _{ij})={\cal{L}}(\phi _{i})$.

Now let $M$ be a bounded distributive lattice and, for all $i\in I$, $f_{i}:{\cal{L}}(A_{i})\rightarrow M$ be bounded lattice morphisms such that, for every $i\leq j$, $f_{j}\circ {\cal{L}}(\phi _{ij})=f_{i}$. Let us define a function $f:{\cal{L}}(A)\rightarrow M$. Let $a\in A$. Then, by the construction of $A$, there exist $i\in I$ and $a_{i}\in A_{i}$ such that $a=[a_{i}]=\phi _{i}(a_{i})$. Set $f(\lambda (a))=f_{i}(\lambda _{i}(a_{i}))$. The surjectivity of $\lambda $ shows that $f$ is completely defined.

Let $a,b\in A$ such that $\lambda (a)=\lambda (b)$. Let $i,j\in I$, $a_{i}\in A_{i}$ and $b_{j}\in A_{j}$ such that $a=[a_{i}]=\phi _{i}(a_{i})$ and $b=[b_{j}]=\phi _{j}(b_{j})$. By condition 5), $\lambda (a)=\lambda (b)$ iff there exist $n,p\in {\rm I\! N}^{*}$ such that $a^{n}\leq b$ and $b^{p}\leq a$. $a^{n}\leq b$ iff $[a_{i}]^{n}\leq [b_{j}]$ iff $\phi _{i}(a_{i})^{n}\leq \phi _{j}(b_{j})$ iff $\phi _{i}(a_{i}^{n})\leq \phi _{j}(b_{j})$ iff there exists $k\in {\rm I\! N}^{*}$ such that $i\leq k$, $j\leq k$ and $\phi _{ik}(a_{i}^{n})\leq \phi _{jk}(b_{j})$. Property c), the commutative diagrams below and the fact that $f_{k}$ and $\lambda _{k}$ are order-preserving show that: $f_{i}(\lambda _{i}(a_{i}))=f_{i}(\lambda _{i}(a_{i}^{n}))=f_{k}({\cal{L}}(\phi _{ik})(\lambda _{i}(a_{i}^{n})))=f_{k}(\lambda _{k}(\phi _{ik}(a_{i}^{n})))\leq f_{k}(\lambda _{k}(\phi _{jk}(b_{j})))=f_{k}({\cal{L}}(\phi _{jk})(\lambda _{j}(b_{j})))=f_{j}(\lambda _{j}(b_{j}))$. Analogously, $b^{p}\leq a$ implies that $f_{j}(\lambda _{j}(b_{j}))\leq f_{i}(\lambda _{i}(a_{i}))$. Therefore $f_{i}(\lambda _{i}(a_{i}))=f_{j}(\lambda _{j}(b_{j}))$, which is equivalent to $f(\lambda (a))=f(\lambda (b))$. Hence $f$ is well defined.

\begin{center}
\begin{picture}(100,90)(0,0)
\put(18,67){$A_{i}$}
\put(30,70){\vector(1,0){40}}
\put(45,75){$\phi _{ik}$}
\put(72,67){$A_{k}$}
\put(1,32){${\cal{L}}(A_{i})$}
\put(72,32){${\cal{L}}(A_{k})$}
\put(30,35){\vector(1,0){40}}
\put(35,40){${\cal{L}}(\phi _{ik})$}
\put(24,65){\vector(0,-1){20}}
\put(13,52){$\lambda _{i}$}
\put(75,65){\vector(0,-1){20}}

\put(77,52){$\lambda _{k}$}
\put(75,30){\vector(0,-1){20}}
\put(77,18){$f_{k}$}

\put(30,30){$\vector(2,-1){40}$}
\put(39,10){$f_{i}$}
\put(73,0){$M$}
\end{picture}
\end{center}

Let $a,b\in A$ and $i,j\in I$, $a_{i}\in A_{i}$ and $b_{j}\in A_{j}$ such that $a=[a_{i}]$ and $b=[b_{j}]$. There exists $k\in I$ such that $i\leq k$ and $j\leq k$. Then $a\vee b=[\phi _{ik}(a_{i})\vee \phi _{jk}(b_{j})]$, with $\phi _{ik}(a_{i})\vee \phi _{jk}(b_{j})\in A_{k}$, so $f(\lambda (a)\vee \lambda (b))=f(\lambda(a\vee b))=f_{k}(\lambda _{k}(\phi _{ik}(a_{i})\vee \phi _{jk}(b_{j})))=f_{k}(\lambda _{k}(\phi _{ik}(a_{i})))\vee f_{k}(\lambda _{k}(\phi _{jk}(b_{j})))=f_{k}({\cal{L}}(\phi _{ik})(\lambda _{i}(a_{i})))\vee f_{k}({\cal{L}}(\phi _{jk})(\lambda _{j}(b_{j})))=f_{i}(\lambda _{i}(a_{i}))\vee f_{j}(\lambda _{j}(b_{j}))=f(\lambda (a))\vee f(\lambda (b))$, by condition 2) and the fact that $f_{k}$, $\phi _{ik}$ and $\phi _{jk}$ are bounded lattice morphisms. Analogously, but using property b) instead of condition 2), we get that $f(\lambda (a)\wedge \lambda (b))=f(\lambda (a))\wedge f(\lambda (b))$. For all $i\in I$, $0=\phi _{i}(0)=[0]$, so $f(0)=f(\lambda (0))=f_{i}(\lambda _{i}(0))=0$, by condition 3) and the fact that $f_{i}$ is a bounded lattice morphism. Analogously, $f(1)=1$. Hence $f$ is a bounded lattice morphism. We have used the surjectivity of $\lambda $.

Now let us prove the uniqueness of $f$. Let $g:{\cal{L}}(A)\rightarrow M$, such that, for all $i\in I$, $g\circ {\cal{L}}(\phi _{i})=f_{i}$.

\begin{center}
\begin{picture}(100,90)(0,0)
\put(18,67){$A_{i}$}
\put(30,70){\vector(1,0){40}}
\put(45,75){$\phi _{i}$}
\put(72,67){$A$}
\put(1,32){${\cal{L}}(A_{i})$}
\put(72,32){${\cal{L}}(A)$}
\put(30,35){\vector(1,0){40}}
\put(35,40){${\cal{L}}(\phi _{i})$}
\put(24,65){\vector(0,-1){20}}
\put(13,52){$\lambda _{i}$}
\put(75,65){\vector(0,-1){20}}
\put(77,52){$\lambda $}
\put(75,30){\vector(0,-1){20}}
\put(77,18){$g$}
\put(30,30){$\vector(2,-1){40}$}

\put(39,10){$f_{i}$}
\put(73,0){$M$}
\end{picture}
\end{center}

The diagrams above are commutative, which justifies the following equalities. Let $a\in A$ and $i\in I$, $a_{i}\in A_{i}$, such that $a=[a_{i}]=\phi _{i}(a_{i})$. Then $g(\lambda (a))=g(\lambda (\phi _{i}(a_{i})))=g({\cal{L}}(\phi _{i})(\lambda _{i}(a_{i})))=f_{i}(\lambda _{i}(a_{i}))=f(\lambda (a))$. By the surjectivity of $\lambda $, we get that $g=f$.\end{proof}

\begin{proposition}
${\cal{L}}$ preserves Boolean powers.

\end{proposition}

\begin{proof}
Let $A$ be a residuated lattice and $B$ a Boolean algebra. By using in turn Theorem \ref{tash}, Proposition \ref{limind} and again Theorem \ref{tash}, we get: ${\cal{L}}(A[B])\cong {\cal{L}}(\lim _{B}(A))\cong \lim _{B}({\cal{L}}(A))\cong {\cal{L}}(A)[B]$.\end{proof}

\section{Stone Algebras}
\label{stonealgebras}

\hspace*{11pt} This section contains other preservation properties of ${\cal{L}}$, along with several properties transferred between ${\cal{D}}01$ and ${\cal{RL}}$ through ${\cal{L}}$.

Concerning Stone and strongly Stone structures (by structure we mean here bounded distributive lattice or residuated lattice), the first question that arises is whether they exist. Naturally, any strongly Stone structure is Stone and any complete Stone structure is strongly Stone. The answer to the question above is given by the fact that the trivial structure is strongly Stone and, moreover, any chain is strongly Stone, because a chain $A$ clearly has all co-annihilators equal to $\{1\}$, except for $1^{\top }$, which is equal to A.

Until mentioned otherwise, let $A$ be a residuated lattice and $({\cal{L}}(A),\lambda )$ its reticulation.

\begin{remark}
For any $a\in A$, we have: $\lambda (a)=1$ iff $a=1$, and $\lambda (a)=0$ iff there exists $n\in {\rm I\! N}^{*}$ such that $a^{n}=0$.
\label{unu}
\end{remark}

\begin{proof}
By conditions 3) and 5), we get: $\lambda (a)=1$ iff $1\leq \lambda (a)$ iff $\lambda (1)\leq \lambda (a)$ iff there exists $n\in {\rm I\! N}^{*}$ such that $1^{n}\leq a$ iff $1\leq a$ iff $a=1$.

Again by conditions 3) and 5), we have: $\lambda (a)=0$ iff $\lambda (a)\leq 0$ iff $\lambda (a)\leq \lambda (0)$ iff there exists $n\in {\rm I\! N}^{*}$ such that $a^{n}\leq 0$ iff there exists $n\in {\rm I\! N}^{*}$ such that $a^{n}=0$.\end{proof}

\begin{remark}
For any subset $X$ of $A$, $\lambda (X^{\top })=\lambda (X)^{\top }$.
\label{comuttop}
\end{remark}

\begin{proof}
By conditions 4) and 2) and Remark \ref{unu}, we have: $\lambda (X)^{\top }=\{\lambda (a)|a\in A,(\forall x\in X)\lambda (a)\vee \lambda (x)=1\}=\{\lambda (a)|a\in A,(\forall x\in X)\lambda (a\vee x)=1\}=\{\lambda (a)|a\in A,(\forall x\in X)a\vee x=1\}=\lambda (X^{\top })$.\end{proof}

\begin{proposition}
$A$ is a Stone residuated lattice iff ${\cal{L}}(A)$ is a Stone lattice.
\label{stoneiff}
\end{proposition}

\begin{proof}
Assume that $A$ is a Stone residuated lattice and let $l\in {\cal{L}}(A)$. $\lambda $ is surjective, hence there exists $a\in A$ with $\lambda (a)=l$. By Definition \ref{stone}, there exists $e\in B(A)$ such that $a^{\top }=<e>$. The fact that $e\in B(A)$ obviously implies that $\lambda (e)\in B({\cal{L}}(A))$ (see also Proposition \ref{izomb}). By Remarks \ref{comuttop} and \ref{lambdagen}, $l^{\top }=\lambda (a)^{\top }=\lambda (a^{\top })=\lambda (<e>)=<\lambda (e)>$. Therefore ${\cal{L}}(A)$ is a Stone lattice.

Now conversely: assume that ${\cal{L}}(A)$ is a Stone lattice and let $a\in A$. By Definition \ref{stone}, the surjectivity of $\lambda $ and Remark \ref{comuttop}, there exists $e\in A$, such that $\lambda (e)\in B({\cal{L}}(A))$ and $\lambda (a^{\top })=\lambda (a)^{\top }=<\lambda (e)>$. The fact that $\lambda (e)\in B({\cal{L}}(A))$ and the surjectivity of $\lambda $ imply that there exists $f\in A$ such that $\lambda (e)\wedge \lambda (f)=0$ and $\lambda (e)\vee \lambda (f)=1$. By condition 1) and Remark \ref{unu}, the following equivalences hold: $\lambda (e)\wedge \lambda (f)=0$ iff $\lambda (e\odot f)=0$ iff there exists $n\in {\rm I\! N}^{*}$ such that $(e\odot f)^{n}=0$. Let us fix one such $n$. Hence $e^{n}\odot f^{n}=0$. By property c), condition 2) and Remark \ref{unu}, we get the following equivalences: $\lambda (e)\vee \lambda (f)=1$ iff $\lambda (e^{n})\vee \lambda (f^{n})=1$ iff $\lambda (e^{n}\vee f^{n})=1$ iff $e^{n}\vee f^{n}=1$. By Remark \ref{`calcul`}, (ii), we get $e^{n}\wedge f^{n}=e^{n}\odot f^{n}=0$. Therefore $e^{n}\in B(A)$. By property c) and Remark \ref{lambdagen}, $\lambda (a^{\top })=<\lambda (e)>=<\lambda (e^{n})>=\lambda (<e^{n}>)$. By Proposition \ref{topfiltru} and Lemma \ref{lambdaegal}, we get $a^{\top }=<e^{n}>$. So $A$ is a Stone residuated lattice.\end{proof}

\begin{proposition}
$A$ is a strongly Stone residuated lattice iff ${\cal{L}}(A)$ is a strongly Stone lattice.
\label{stronglystoneiff}
\end{proposition}

\begin{proof}
Similar to the proof of Proposition \ref{stoneiff}.\end{proof}

\begin{proposition}
Let $A$ be a residuated lattice. Then ${\rm Co-Ann}(A)$ and ${\rm Co-Ann}({\cal{L}}(A))$ are isomorphic Boolean algebras.
\end{proposition}

\begin{proof}
Let $({\cal{L}}(A),\lambda )$ be the reticulation of $A$ and $\mu :{\rm Co-Ann}(A)\rightarrow {\rm Co-Ann}({\cal{L}}(A))$, for all $F\in {\rm Co-Ann}(A)$, $\mu (F)=\lambda (F)$. Proposition \ref{`izomf`} and Remark \ref{comuttop} show that $\mu $ is an injective morphism of Boolean algebras. For all $F\in {\rm Co-Ann}({\cal{L}}(A))$, there exists $X\subseteq {\cal{L}}(A)$ such that $F=X^{\top }$. By the surjectivity of $\lambda $, there exists $Y\subseteq A$ such that $\lambda (Y)=X$. $Y^{\top }\in {\rm Co-Ann}(A)$ and, by Remark \ref{comuttop}, $\mu (Y^{\top })=\lambda (Y^{\top })=\lambda (Y)^{\top }=X^{\top }=F$. So $\mu $ is also surjective, hence it is a Boolean isomorphism.\end{proof}

\begin{theorem}
Let $A$ be a residuated lattice and $m$ an infinite cardinal. Then the following are equivalent:

\noindent (i) for each subset $X$ of $A$ with $|X|\leq m$, there exists an element $e\in B(A)$ such that $X^{\top }=<e>$;

\noindent (ii) $A$ is a Stone residuated lattice and $B(A)$ is an $m$-complete Boolean algebra;

\noindent (iii) $A^{\top \top }=\{(a^{\top })^{\top }|a\in A\}$ is an $m$-complete Boolean sublattice of ${\cal{F}}(A)$;

\noindent (iv) for all $a,b\in A$, $(a\wedge b)^{\top }=a^{\top }\vee b^{\top }$ and, for each subset $X$ of $A$ with $|X|\leq m$, there exists an element $x\in A$ such that $X^{\top \top}=x^{\top }$;

\noindent (v) for each subset $X$ of $A$ with $|X|\leq m$, $X^{\top }\vee X^{\top \top }=A$.
\label{caractlrstone}

\end{theorem}

\begin{proof}
Let $({\cal{L}}(A),\lambda )$ be the reticulation of $A$. By Theorem \ref{caractlstone}, it is sufficient to prove that condition (i) for $A$ is equivalent with condition (i) for ${\cal{L}}(A)$ and the same is valid for conditions (ii)-(v).

So let us denote the following conditions:

\noindent (i-$A$) for each subset $X$ of $A$ with $|X|\leq m$, there exists an element $e\in B(A)$ such that $X^{\top }=<e>$,

\noindent (i-${\cal{L}}(A)$) for each subset $X$ of ${\cal{L}}(A)$ with $|X|\leq m$, there exists an element $e\in B({\cal{L}}(A))$ such that $X^{\top }=<e>$,

\noindent and let us prove that (i-$A$) is equivalent to (i-${\cal{L}}(A)$).

First, let us assume that (i-$A$) is satisfied and let $X\subseteq {\cal{L}}(A)$ with $|X|\leq m$. The fact that $\lambda $ is surjective implies that there exists $Y\subseteq A$ with $|Y|=|X|\leq m$ and $\lambda (Y)=X$. By (i-$A$), there exists $e\in B(A)$ such that $Y^{\top }=<e>$. Obviously, $\lambda (e)\in B({\cal{L}}(A))$ (see also Proposition \ref{izomb}). By Remarks \ref{comuttop} and \ref{lambdagen}, $X^{\top }=\lambda (Y)^{\top }=\lambda (Y^{\top })=\lambda (<e>)=<\lambda (e)>$.

Now let us assume that (i-${\cal{L}}(A)$) is satisfied and let $X\subseteq A$ with $|X|\leq m$. Then $|\lambda (X)|\leq |X|\leq m$, so there exists $f\in B({\cal{L}}(A))$ such that $\lambda (X)^{\top }=<f>$. $\lambda $ is surjective, so there exists $e\in A$ such that $\lambda (e)=f$. As in the proof of Proposition \ref{stoneiff}, it follows that there exists $n\in {\rm I\! N}^{*}$ such that $e^{n}\in B(A)$. Using Remarks \ref{comuttop} and \ref{lambdagen} and property c), we get $\lambda (X^{\top })=\lambda (X)^{\top }=<\lambda (e)>=<\lambda (e^{n})>=\lambda (<e^{n}>)$, which, by Proposition \ref{topfiltru} and Lemma \ref{lambdaegal}, implies $X^{\top }=<e^{n}>$.

We denote:

\noindent (ii-$A$) $A$ is a Stone residuated lattice and $B(A)$ is an $m$-complete Boolean algebra;

\noindent (ii-${\cal{L}}(A)$) ${\cal{L}}(A)$ is a Stone lattice and $B({\cal{L}}(A))$ is an $m$-complete Boolean algebra.

Propositions \ref{stoneiff} and \ref{izomb} ensure us that (ii-$A$) and (ii-${\cal{L}}(A)$) are equivalent.

Let us denote:

\noindent (iii-$A$) $A^{\top \top }=\{(a^{\top })^{\top }|a\in A\}$ is an $m$-complete Boolean sublattice of ${\cal{F}}(A)$;

\noindent (iii-${\cal{L}}(A)$) ${\cal{L}}(A)^{\top \top }=\{(l^{\top })^{\top }|l\in {\cal{L}}(A)\}$ is an $m$-complete Boolean sublattice of ${\cal{F}}({\cal{L}}(A))$.

Let $\psi :A^{\top \top }\rightarrow {\cal{L}}(A)^{\top \top }$, for all $a\in A$, $\psi (a^{\top \top })=\lambda (a)^{\top \top }=\lambda (a^{\top \top})$, where the last equality was obtained from Remark \ref{comuttop}. By Propositions \ref{topfiltru} and \ref{`izomf`}, $\psi $ is an injective morphism of bounded lattices. The fact that $\lambda $ is surjective implies that $\psi $ is surjective. Hence $\psi $ is a bounded lattice isomorphism.

Let $L:{\cal{F}}(A)\rightarrow {\cal{F}}({\cal{L}}(A))$ be the bounded lattice isomorphism from Proposition \ref{`izomf`}: for all $F\in {\cal{F}}(A)$, $L(F)=\lambda (F)$. If there exists an injective morphism of bounded lattices $f:A^{\top \top }\rightarrow {\cal{F}}(A)$, then the function $g:{\cal{L}}(A)^{\top \top }\rightarrow {\cal{F}}({\cal{L}}(A))$, defined by $g=L\circ f\circ \psi ^{-1}$, is an injective morphism of bounded lattices. If there exists an injective morphism of bounded lattices $g:{\cal{L}}(A)^{\top \top }\rightarrow {\cal{F}}({\cal{L}}(A))$, then the function $f:A^{\top \top }\rightarrow {\cal{F}}(A)$, defined by $f=L^{-1}\circ g\circ \psi $, is an injective morphism of bounded lattices.

\begin{center}
\begin{picture}(120,60)(0,0)
\put(15,0){${\cal{F}}(A)$}
\put(43,4){\vector(1,0){30}}
\put(55,6){$L$}
\put(75,0){${\cal{F}}({\cal{L}}(A))$}
\put(15,40){$A^{\top \top}$}
\put(19,38){\vector(0,-1){28}}
\put(11,21){$f$}
\put(40,43){\vector(1,0){33}}
\put(55,45){$\psi $}

\put(75,40){${\cal{L}}(A)^{\top \top}$}
\put(90,38){\vector(0,-1){28}}
\put(93,21){$g$}
\end{picture}
\end{center}

The above show the equivalence between (iii-$A$) and (iii-${\cal{L}}(A)$).

We denote:

\noindent (iv-$A$) for all $a,b\in A$, $(a\wedge b)^{\top }=a^{\top }\vee b^{\top }$ and, for each subset $X$ of $A$ with $|X|\leq m$, there exists an element $x\in A$ such that $X^{\top \top}=x^{\top }$;

\noindent (iv-${\cal{L}}(A)$) for all $l,p\in {\cal{L}}(A)$, $(l\wedge p)^{\top }=l^{\top }\vee p^{\top }$ and, for each subset $X$ of ${\cal{L}}(A)$ with $|X|\leq m$, there exists an element $x\in {\cal{L}}(A)$ such that $X^{\top \top}=x^{\top }$.

Let us assume that (iv-$A$) is satisfied and let $a,b\in A$. We will use the surjectivity of $\lambda $. By property b), Remark \ref{comuttop} and condition 2), $(\lambda(a)\wedge \lambda (b))^{\top }=\lambda (a\wedge b)^{\top }=\lambda ((a\wedge b)^{\top })=\lambda (a^{\top }\vee b^{\top })=\lambda (a^{\top })\vee \lambda (b^{\top })=\lambda (a)^{\top }\vee \lambda (b)^{\top }$. Let $X\subseteq {\cal{L}}(A)$ with $|X|\leq m$. By the surjectivity of $\lambda $, there exists $Y\subseteq A$ with $\lambda (Y)=X$ and $|Y|=|X|\leq m$. This implies that there exists $y\in A$ such that $Y^{\top \top}=y^{\top }$, which in turn, by Remark \ref{comuttop}, implies that $X^{\top \top}=\lambda (Y)^{\top \top }=\lambda (Y^{\top \top})=\lambda (y^{\top })=\lambda (y)^{\top }$.

Now let us assume that (iv-${\cal{L}}(A)$) is satisfied and let $a,b\in A$. We have: $(\lambda(a)\wedge \lambda (b))^{\top }=\lambda (a)^{\top }\vee \lambda (b)^{\top }$, which, by computations similar to the ones above, is equivalent to: $\lambda ((a\wedge b)^{\top })=\lambda (a^{\top }\vee b^{\top })$. This, by Proposition \ref{topfiltru} and Lemma \ref{lambdaegal}, implies that $(a\wedge b)^{\top }=a^{\top }\vee b^{\top }$. Let $Y\subseteq A$ with $|Y|\leq m$. Then $|\lambda (Y)|\leq |Y|\leq m$, so, by the surjectivity of $\lambda $, there exists $y\in A$ such that $\lambda (Y)^{\top \top }=\lambda (y)^{\top }$. By computations similar to the ones above, this is equivalent to $\lambda (Y^{\top \top })=\lambda (y^{\top })$, which, by Proposition \ref{topfiltru} and Lemma \ref{lambdaegal}, is equivalent to $Y^{\top \top }=y^{\top }$.

We denote:

\noindent (v-$A$) for each subset $X$ of $A$ with $|X|\leq m$, $X^{\top }\vee X^{\top \top }=A$;

\noindent (v-${\cal{L}}(A)$) for each subset $X$ of ${\cal{L}}(A)$ with $|X|\leq m$, $X^{\top }\vee X^{\top \top }={\cal{L}}(A)$.

Let us assume that (v-$A$) is satisfied. Let $X\subseteq {\cal{L}}(A)$ such that $|X|\leq m$. The surjectivity of $\lambda $ implies that there exists $Y\subseteq A$ with $\lambda (Y)=X$ and $|Y|=|X|\leq m$. Therefore $Y^{\top }\vee Y^{\top \top }=A$. By Remark \ref{comuttop}, Proposition \ref{`izomf`} and the surjectivity of $\lambda $ (which is actually implied by Proposition \ref{`izomf`}), this implies that $X^{\top }\vee X^{\top \top }=\lambda (Y)^{\top }\vee \lambda (Y)^{\top \top }=\lambda (Y^{\top })\vee \lambda (Y^{\top \top })=\lambda (Y^{\top }\vee Y^{\top \top })=\lambda (A)={\cal{L}}(A)$.

Conversely, let us assume that (v-${\cal{L}}(A)$) is satisfied. Let $Y\subseteq A$ such that $|Y|\leq m$. Then $|\lambda (Y)|\leq |Y|\leq m$, so $\lambda (Y)^{\top }\vee \lambda (Y)^{\top \top }={\cal{L}}(A)$. By computations similar to the ones above, this is equivalent to $\lambda (Y^{\top }\vee Y^{\top \top })=\lambda (A)$, which, by Lemma \ref{lambdaegal}, is equivalent to $Y^{\top }\vee Y^{\top \top }=A$.\end{proof}

A residuated lattice will be called an {\em $m$-Stone residuated lattice} iff the conditions of Theorem \ref{caractlrstone} hold for it.

\begin{proposition}
$A$ is an $m$-Stone residuated lattice iff ${\cal{L}}(A)$ is an $m$-Stone lattice.
\label{mstoneiff}
\end{proposition}
\begin{proof}
By the proof of Theorem \ref{caractlrstone}.\end{proof}

The following two remarks show that Stone residuated lattices do not have a characterization like the one in \cite[Theorem 8.7.1, page 164]{bal} for Stone pseudocomplemented distributive lattices.

\begin{remark}
There exist Stone residuated lattices $A$ with elements $a\in A$ that do not satisfy the identity $\neg \, a\vee \neg \, \neg \, a=1$.
\label{stonedar}

\end{remark}
\begin{proof}
Let us consider the following residuated lattice: $A=\{0,a,b,c,1\}$, with the structure described below. This is an example of residuated lattice from \cite[Section 11.1]{ior1}, which can also be found in \cite{ior}.

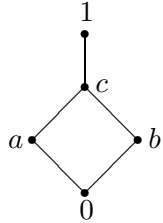
\begin{center}
\begin{picture}(70,95)(0,0)
\put(30,10){\line(-1,1){20}}
\put(30,10){\line(1,1){20}}
\put(30,50){\line(-1,-1){20}}
\put(30,50){\line(1,-1){20}}
\put(30,50){\line(0,1){20}}
\put(30,10){\circle*{3}}
\put(10,30){\circle*{3}}
\put(50,30){\circle*{3}}
\put(30,50){\circle*{3}}

\put(30,70){\circle*{3}}
\put(28,0){$0$}
\put(1,27){$a$}
\put(54,27){$b$}
\put(34,48){$c$}
\put(28,75){$1$}
\end{picture}
\end{center}

\begin{center}
\begin{tabular}{c|ccccc}
$\rightarrow $ & $0$ & $a$ & $b$ & $c$ & $1$ \\ \hline
$0$ & $1$ & $1$ & $1$ & $1$ & $1$ \\
$a$ & $b$ & $1$ & $b$ & $1$ & $1$ \\
$b$ & $a$ & $a$ & $1$ & $1$ & $1$ \\
$c$ & $0$ & $a$ & $b$ & $1$ & $1$ \\
$1$ & $0$ & $a$ & $b$ & $c$ & $1$
\end{tabular}
\end{center}

\noindent and $\odot =\wedge $.

$B(A)=\{0,1\}$, $<0>=A$, $<1>=\{1\}$, $0^{\top }=a^{\top }=b^{\top }=c^{\top }=\{1\}$, $1^{\top }=A$, therefore $A$ is a Stone residuated lattice. But $\neg \, a=b$, $\neg \, \neg \, a=\neg \, b=a$, so $\neg \, a\vee \neg \, \neg \, a=b\vee a=c\neq 1$.\end{proof}

Notice that $A$ from the proof above is strongly Stone.

\begin{remark}
There exist residuated lattices $A$ that satisfy the identity $\neg \, a\vee \neg \, \neg \, a=1$ for all $a\in A$ and that are not Stone.
\end{remark}
\begin{proof}

Let $A=\{0,n,a,b,i,f,g,h,j,c,d,1\}$, with the residuated lattice structure presented below. This is an example of residuated lattice from \cite[Section 15.2.1]{ior2}, which can also be found in \cite{ior}.

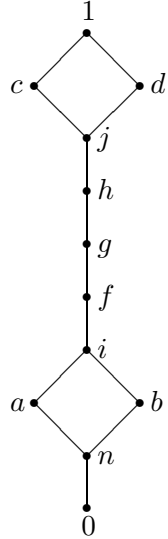
\begin{center}
\begin{picture}(70,210)(0,0)
\put(30,10){\line(0,1){20}}

\put(30,30){\line(1,1){20}}
\put(30,30){\line(-1,1){20}}
\put(30,70){\line(1,-1){20}}
\put(30,70){\line(-1,-1){20}}
\put(30,70){\line(0,1){80}}
\put(30,150){\line(1,1){20}}
\put(30,150){\line(-1,1){20}}
\put(30,190){\line(1,-1){20}}
\put(30,190){\line(-1,-1){20}}
\put(30,10){\circle*{3}}
\put(10,50){\circle*{3}}

\put(30,30){\circle*{3}}
\put(50,50){\circle*{3}}
\put(30,70){\circle*{3}}
\put(30,90){\circle*{3}}
\put(30,110){\circle*{3}}
\put(30,130){\circle*{3}}

\put(30,150){\circle*{3}}
\put(30,190){\circle*{3}}
\put(10,170){\circle*{3}}
\put(50,170){\circle*{3}}
\put(28,0){$0$}
\put(34,27){$n$}
\put(34,67){$i$}
\put(34,87){$f$}
\put(34,107){$g$}
\put(34,127){$h$}
\put(34,147){$j$}
\put(1,47){$a$}
\put(54,47){$b$}

\put(1,167){$c$}
\put(54,167){$d$}
\put(28,195){$1$}
\end{picture}
\end{center}

\begin{center}
\begin{tabular}{c|cccccccccccc}
$\rightarrow $ & $0$ & $n$ & $a$ & $b$ & $i$ & $f$ & $g$ & $h$ & $j$ & $c$ & $d$ & $1$ \\ \hline
$0$ & $1$ & $1$ & $1$ & $1$ & $1$ & $1$ & $1$ & $1$ & $1$ & $1$ & $1$ & $1$ \\
$n$ & $0$ & $1$ & $1$ & $1$ & $1$ & $1$ & $1$ & $1$ & $1$ & $1$ & $1$ & $1$ \\
$a$ & $0$ & $d$ & $1$ & $d$ & $1$ & $1$ & $1$ & $1$ & $1$ & $1$ & $1$ & $1$ \\
$b$ & $0$ & $c$ & $c$ & $1$ & $1$ & $1$ & $1$ & $1$ & $1$ & $1$ & $1$ & $1$ \\
$i$ & $0$ & $j$ & $c$ & $d$ & $1$ & $1$ & $1$ & $1$ & $1$ & $1$ & $1$ & $1$ \\
$f$ & $0$ & $h$ & $h$ & $h$ & $h$ & $1$ & $1$ & $1$ & $1$ & $1$ & $1$ & $1$ \\
$g$ & $0$ & $g$ & $g$ & $g$ & $g$ & $h$ & $1$ & $1$ & $1$ & $1$ & $1$ & $1$ \\
$h$ & $0$ & $f$ & $f$ & $f$ & $f$ & $h$ & $h$ & $1$ & $1$ & $1$ & $1$ & $1$ \\
$j$ & $0$ & $i$ & $i$ & $i$ & $i$ & $f$ & $g$ & $h$ & $1$ & $1$ & $1$ & $1$ \\
$c$ & $0$ & $b$ & $i$ & $b$ & $i$ & $f$ & $g$ & $h$ & $d$ & $1$ & $d$ & $1$ \\

$d$ & $0$ & $a$ & $a$ & $i$ & $i$ & $f$ & $g$ & $h$ & $c$ & $c$ & $1$ & $1$ \\
$1$ & $0$ & $n$ & $a$ & $b$ & $i$ & $f$ & $g$ & $h$ & $j$ & $c$ & $d$ & $1$ \end{tabular}
\end{center}

\begin{center}
\begin{tabular}{c|cccccccccccc}

$\odot $ & $0$ & $n$ & $a$ & $b$ & $i$ & $f$ & $g$ & $h$ & $j$ & $c$ & $d$ & $1$ \\ \hline
$0$ & $0$ & $0$ & $0$ & $0$ & $0$ & $0$ & $0$ & $0$ & $0$ & $0$ & $0$ & $0$ \\
$n$ & $0$ & $n$ & $n$ & $n$ & $n$ & $n$ & $n$ & $n$ & $n$ & $n$ & $n$ & $n$ \\
$a$ & $0$ & $n$ & $n$ & $n$ & $n$ & $n$ & $n$ & $n$ & $n$ & $a$ & $n$ & $a$ \\
$b$ & $0$ & $n$ & $n$ & $n$ & $n$ & $n$ & $n$ & $n$ & $n$ & $n$ & $b$ & $b$ \\
$i$ & $0$ & $n$ & $n$ & $n$ & $n$ & $n$ & $n$ & $n$ & $n$ & $a$ & $b$ & $i$ \\
$f$ & $0$ & $n$ & $n$ & $n$ & $n$ & $n$ & $n$ & $n$ & $f$ & $f$ & $f$ & $f$ \\
$g$ & $0$ & $n$ & $n$ & $n$ & $n$ & $n$ & $n$ & $f$ & $g$ & $g$ & $g$ & $g$ \\
$h$ & $0$ & $n$ & $n$ & $n$ & $n$ & $n$ & $f$ & $f$ & $h$ & $h$ & $h$ & $h$ \\
$j$ & $0$ & $n$ & $n$ & $n$ & $n$ & $f$ & $g$ & $h$ & $j$ & $j$ & $j$ & $j$ \\
$c$ & $0$ & $n$ & $a$ & $n$ & $a$ & $f$ & $g$ & $h$ & $j$ & $c$ & $j$ & $c$ \\
$d$ & $0$ & $n$ & $n$ & $b$ & $b$ & $f$ & $g$ & $h$ & $j$ & $j$ & $d$ & $d$ \\
$1$ & $0$ & $n$ & $a$ & $b$ & $i$ & $f$ & $g$ & $h$ & $j$ & $c$ & $d$ & $1$ \end{tabular}
\end{center}

$A$ satisfies the identity in the enunciation. $B(A)=\{0,1\}$, $<0>=A$, $<1>=\{1\}$, but $c^{\top }=\{d,1\}$, hence $A$ is not Stone.\end{proof}

\begin{remark}
There exist residuated lattices $A$ that do not satisfy the identity $\neg \, a\vee \neg \, \neg \, a=1$ for all $a\in A$, but whose reticulation ${\cal{L}}(A)$ is a pseudocomplemented lattice and satisfies this identity: $l^{*}\vee l^{**}=1$ for all $l\in {\cal{L}}(A)$.
\label{nucdefstone}
\end{remark}
\begin{proof}
Let us consider the example from the proof of Remark \ref{stonedar}. In this residuated lattice, that we denote by $A$, $\odot =\wedge $, hence, as shown by \cite[Proposition 3.1]{eu2}, $A$ and ${\cal{L}}(A)$ are isomorphic bounded lattices. One can see that, therefore, ${\cal{L}}(A)$ is pseudocomplemented. As we have seen in Remark \ref{stonedar}, $A$ is Stone, so ${\cal{L}}(A)$ is also Stone, which can be seen from Proposition \ref{stoneiff} or from the fact that these two algebras are isomorphic as bounded lattices and the fact that $\odot =\wedge $ in $A$, which ensures us that their principal filters coincide. Since ${\cal{L}}(A)$ is Stone, it follows by \cite[Theorem 8.7.1, page 164]{bal} that it satisfies the identity in the enunciation. But, as shown in Remark \ref{stonedar}, $A$ does not satisfy this identity.\end{proof}

The remark above shows that the alternate definition of Stone algebras, from \cite{rcig}, is not transferrable through the reticulation, which is the reason why we have chosen our definition over it.

\section{Acknowledgements}

\hspace*{11pt} I thank Professor George Georgescu for offering me this research theme and for numerous suggestions concerning the tackling of these mathematical problems.

\end{document}